\documentclass[11pt]{amsart}

\usepackage{epsfig,color}

\usepackage{amsmath}
\usepackage{amsthm}
\usepackage{hyperref}
\usepackage[margin=1.0in]{geometry}

\numberwithin{equation}{section}

\newtheorem{prop}{Proposition}
\newtheorem{lemma}[prop]{Lemma}

\newtheorem{thm}[prop]{Theorem}
\newtheorem{cor}[prop]{Corollary}
\newtheorem{conj}[prop]{Conjecture}

\numberwithin{prop}{section}

\theoremstyle{definition}
\newtheorem{defn}[prop]{Definition}

\newtheorem{rmk}[prop]{Remark}

\newcommand{\del}{\partial}
\newcommand{\dt}{\frac{\partial}{\partial t}}
\newcommand{\brs}[1]{\left| #1 \right|}

\newcommand{\gG}{\Gamma}

\newcommand{\gs}{\sigma}

\newcommand{\gl}{\lambda}

\newcommand{\ga}{\alpha}
\newcommand{\gb}{\beta}

\newcommand{\N}{\nabla}

\newcommand{\CC}{\mathcal C}

\newcommand{\til}[1]{\widetilde{#1}}

\renewcommand{\bar}[1]{\overline{#1}}

\newcommand{\bv}{\bar{v}}

\newcommand{\IP}[1]{\left<#1\right>}

\newcommand{\Cp}{\mathcal{C}^{+}}
\newcommand{\Cm}{\mathcal{C}_m^{+}}
\newcommand{\Ca}{\mathcal{C}_{(A)}^{+}}

\DeclareMathOperator{\Rc}{Rc}

\DeclareMathOperator{\grad}{grad}

\newcommand{\dtu}{\frac{\del u}{\del t}}

\begin{document}

\title{Variational structure of the $v_{\frac{n}{2}}$-Yamabe problem}

\author{Matthew Gursky}
\address{Department of Mathematics
         University of Notre Dame\\
         Notre Dame, IN 46556}
\email{\href{mailto:mgursky@nd.edu}{mgursky@nd.edu}}

\author{Jeffrey Streets}
\address{Department of Mathematics\\
         University of California\\
         Irvine, CA  92617}
\email{\href{mailto:jstreets@math.uci.edu}{jstreets@math.uci.edu}}

\date{\today}

\thanks{M. Gursky gratefully acknowledges support from the NSF via DMS-1509633.
J. Streets gratefully acknowledges support from the NSF via DMS-1454854
and from the
Alfred P. Sloan Foundation via a Sloan Research Fellowship.}

\begin{abstract} We define a formal Riemannian metric on a
conformal class in the context of the $v_{\frac{n}{2}}$-Yamabe problem.  We also
give a new variational description of this problem, and show that the associated
functional is geodesically convex.  Formal properties of the negative gradient flow are
also described. These results parallel our work in two dimensions on the Liouville energy and
the uniformization of surfaces \cite{GS_Surfaces}, and our work in four dimensions on the $\sigma_2$-Yamabe problem \cite{GS2}.
\end{abstract}

\maketitle

\section{Introduction}

In \cite{GS_Surfaces} and \cite{GS2} we defined a formal Riemannian metric on the space of conformal metrics satisfying some notion of `positivity'.  In the case of surfaces this condition corresponded to positive Gauss curvature.  In four dimensions, given a conformal class $[g_0]$ we considered the subset
\begin{align} \label{P2C}
\mathcal{C}^{+} = \mathcal{C}^{+}( [g_0]) = \{ g_u = e^{-2u} g_0\ :\ A_u \in \Gamma_2^{+} \},
\end{align}
where $A_u$ is the Schouten tensor of the metric $g_u$ and $\Gamma^{+}_2$ is the positive $2$-cone.  Recall the Schouten tensor
is defined by
\begin{align} \label{Adef}
A = \frac{1}{n-2} \big( Ric - \frac{1}{2(n-1)} R g \big),
\end{align}
where $Ric$ and $R$ are the Ricci and scalar curvatures of $g$, and
\begin{align*}
A_g \in \Gamma_k^{+} \ \ \Longleftrightarrow\ \  \sigma_1(g^{-1} A_g) > 0,\dots, \sigma_k(g^{-1} A_g) > 0,
\end{align*}
where $\sigma_k(\cdot)$ is the elementary symmetric polynomial of degree $k$ of the endomorphism $g^{-1}A$ (see the Introduction of \cite{GS2}).  To simplify notation in what follows we will suppress the inverse of the metric $g_u$ and only write $A_g, A_u$, etc.  Assuming $\mathcal{C}^{+}$ is non-empty we defined the Riemannian metric on $\Cp$ by
\begin{align} \label{gskmetric}
 \IP{ \phi, \psi}_u =&\ \int_M \phi \psi \gs_{2}( A_u) dV_u.
\end{align}
Here we are using the natural identification of the tangent space to $\mathcal{C}^{+}$ at any point with $C^{\infty}(M)$.   Endowed with this metric $\Cp$ enjoys a number of nice formal properties; e.g., $\Cp$ has non-positive sectional curvature.

Of primary interest in our analysis were the variational properties of the functional $F : \Cp \rightarrow \mathbb{R}$ introduced by Chang-Yang \cite{ChangYangMoserVol}, whose critical points are conformal metrics satisfying
\begin{align} \label{s2Y}
\sigma_2(A_u) = const.,
\end{align}
i.e., $g_u$ is a solution of the {\em $\sigma_2$-Yamabe problem}.  In particular, we showed that the functional $F$ is geodesically convex and used this fact to prove a remarkable geometric consequence: solutions of (\ref{s2Y}) are {\em unique}, unless $(M^4,g_0)$ is conformally equivalent to the sphere.   This is a surprising departure from the classical Yamabe problem, where explicit examples of non-uniqueness are known.

To extend these results to higher dimensions $n \geq 6$, it would seem natural to consider the set of conformal metrics $\Cp = \{ g_u = e^{-2u}g_0\ :\ A_u \in \Gamma_{n/2}^{+} \}$, and define the inner product on the tangent space by
\begin{align} \label{IPk}
 \IP{ \phi, \psi}_u =&\ \int_M \phi \psi \sigma_{n/2}(A_u) dV_u.
\end{align}
However, defining the metric in this way introduces a number of technical issues which can be traced back to the fact that in dimensions $n \geq 6$, the quantity $\sigma_{n/2}(A_u)$ lacks the kind of `divergence structure' that it enjoys in dimension four, or the Gauss curvature enjoys in dimension two.  More concretely, the integral
\begin{align} \label{totalSk}
\int \sigma_{n/2}(A_g)\ dV_g
\end{align}
is conformally invariant when $n = 4$, but not in general when $n \geq 6$.  Indeed, this lack of divergence structure is a fundamental difficulty in the study of the {\em $\sigma_k$-Yamabe problem}, which asks whether it is possible to find a conformal metric $g_u = e^{-2u}g$ for which $\sigma_k(A_u) = const.$, assuming $A_g \in \Gamma_k^{+}$.

If $(M^n,g)$ is locally conformally flat (LCF) then the integral (\ref{totalSk}) is conformally invariant (\cite{BG}, \cite{JeffThesis}).  However, it follows from the work of Guan-Viaclovsky \cite{GuanVia} that when $A_g \in \Gamma_{n/2}^{+}$, then the Ricci curvature of $g$ is positive.  Consequently, by Kuiper's Theorem, in the LCF setting the space of conformal metrics $g_u = e^{-2u}g$ with $A_u \in \Gamma_{n/2}^{+}$  will be non-empty only when $(M^n,g)$ is conformally equivalent to the round sphere (or real projective space).  Consequently, imposing the LCF condition for metrics whose Schouten tensor is in $\Gamma_{n/2}^{+}$ is too restrictive.

What is needed in higher dimensions is a conformally invariant quantity of the correct weight which does not require the LCF condition.  Such a quantity appears in the consideration of the renormalized volume of Poincare-Einstein manifolds, see \cite{GrahamRNV}.  To explain this, we briefly recall some definitions.

Let $X$ be the interior of a compact manifold with boundary $\overline{X}$ of dimension $n+1$, and let $M  = \partial X$ denote the boundary.  A metric $g_{+}$ defined on $X$ is said to be {\em conformally compact} if there is a defining function $r \in C^{\infty}(X)$ with $r > 0$ and $dr \neq 0$ on $\partial X$, such that $r^2 g_{+}$ extends to a metric
$\overline{g}$ on $\overline{X}$.  Since we can multiply $r$ by any smooth positive function on $\overline{X}$, a conformally compact metric naturally defines a conformal class of metrics $[ g = \overline{g}|_{M} ]$ on $M = \partial X$, called the {\em conformal infinity} of $(X,g_{+})$.  If in addition $g_{+}$ satisfies the Einstein condition, which we normalize by
\begin{align*}
Ric(g_{+}) = -n g_{+},
\end{align*}
then we say that $(X,g_{+})$ is a {\em Poincar\'e-Eintein} (P-E) manifold.  If $r$ is a defining function such that $|dr|_{\bar{g}} = 1$ on $M = \partial X$ (referred to a as a {\em special} defining function), then $g_{+}$ can be written
\begin{align} \label{gexp}
g_{+} = r^{-2} ( dr^2 + g_r ),
\end{align}
where $g_r$ is a 1-parameter family of metrics on $M$ with $g_0 = g$.  When $n$ is even, Graham \cite{GrahamRNV} showed that
\begin{align}   \label{FGeven}
g_r = g^{(0)} + g^{(2)}r^2 + \cdots + g^{(n)}r^n  + h r^n \log r + \cdots,
\end{align}
where $g^{(0)} = g$ is the induced metric on $M$, the coefficients are formally determined by the conformal representative up to order $n-2$, and $h$ is also formally determined. Using this expansion, Graham gave an expansion for the volume form
\begin{align} \label{volform}
\Big( \dfrac{ \det(g_r) }{\det g}\Big) \sim 1 + \sum_{k \geq 1} v_k r^k,
\end{align}
where $v_k = v_k(g)$ are defined for $1 \leq k \leq n/2$ in general, but are defined for all $k \geq 1$ when $(M,g)$ is LCF.  When $k = 1, 2$, then $v_k(g) = \sigma_k(A)$, while in the LCF case this holds for $k \geq 3$.  Moreover, when $k = n/2$, then
\begin{align} \label{vb}
v = \int v_{n/2}(g) dV_g
\end{align}
is a conformal invariant.  Noting these parallels, Chang-Fang \cite{ChangFang} proposed the study of the functionals $g \mapsto \int v_k(g) dV_g$ for $k < n/2$ as the natural generalization of the functionals given by the
integrals of $\sigma_{k}(A)$ in the non-LCF setting.  This leads to another generalization of the Yamabe problem, the {\em $v_k$-Yamabe problem}, to find (under suitable conditions) in a given conformal class a critical point of the functional $g \mapsto \int v_k(g) dV_g$; i.e., a conformal metric for which $v_k$ is constant.  Later Graham
showed \cite{Graham} further structure of these quantities via their
relationship to ``extended obstruction tensors.''

These results suggest the following natural extension of the formal Riemannian structure given in our earlier work:  Let $(M^{2m},g)$ be a closed Riemannian manifold of even dimension $n = 2m$.  Given a conformal metric
$g_u = e^{-2u}g$, define the inner product
\begin{align} \label{vkmetric}
\IP{\ga,\gb}_u =&\ \int_{M^{2m}} \ga \gb v_m(g_u) dV_u.
\end{align}
For the inner product
above to be positive definite, it is clear that positivity of $v_k(g_u)$ is
required, hence we restrict our conformal metrics to the set
\begin{align} \label{Gammakdef}
 \Cm := \{ u \in C^{\infty}(M)\ |\ v_m(g_u) > 0,\ L > 0 \}.
\end{align}
Here $L$ is the principal symbol of the linearization of $v_m$ (see
\cite{Graham} Theorem 1.5, recorded as Theorem \ref{grahamthm} below).  Although the addition of this assumption may seem superfluous, it will be
important when verifying certain properties of the metric defined by (\ref{vkmetric}).  Moreover, it is easy to see that when the dimension is four this set corresponds to the positive cone defined in (\ref{P2C}).

With these definitions we can now state the objectives of the paper.  Our first goal is to verify some of the basic formal properties of the metric defined in
(\ref{vkmetric}), and to write down the induced connection.  This is carried out in \S \ref{MC}.  For these results we only need to assume that our conformal metrics are in the set $\Cm$.  Next, we prove the existence of a functional $F : [g] \rightarrow \mathbb{R}$ generalizing he functional of Chang-Yang, whose critical points correspond to conformal metrics satisfying
\begin{align} \label{vkY}
v_m(g_u) = const.
\end{align}
(compare with (\ref{s2Y})).  Our construction is an adaptation of the method of Brendle-Viaclovsky \cite{BV}, who gave another proof of the Chang-Yang construction in \cite{ChangYangMoserVol}.   In analogy with our work in low dimensions, we would like to verify that $F$ is geodesically convex.

Here we encounter the crucial point that to understand the variational properties of $F$, we need to restrict to a smaller set of conformal metrics: even in the LCF case, $\Cm$ is strictly larger than the positive $n/2$-cone, so we do not expect $F$ to have nice properties on $\Cm$.  A shift in perspective is warranted, and instead of defining cones by imposing positivity
conditions on various curvature quantities, we instead consider metrics which verify an {\em Andrew's-type inequality}:
\begin{align*}
n \Big[ \int_M \phi^2 dV_g - V_{g}^{-1} \big( \int_M \phi dV_g \big)^2 \Big] &\leq \int_{M} \frac{1}{v_m(g)} L^{ij} \nabla_i \phi \nabla_j \phi dV_g,  \ \ \ \ \mbox{(A)}
\end{align*}
where $V_g$ is the volume of $g$, and equality holds if and only if $\phi = 0$ or $(M^{2m},g)$ is conformally equivalent to the round sphere and $\phi$ is a first-order spherical harmonic.  When the dimension $n =4$ and $A_g \in \Gamma_2^{+}$, this inequality is a consequence of the sharp Poincar\'e-type inequality of Andrews \cite{Andrews}, and played a key role in our previous work.  In dimensions $n = 2m \geq 6$, given a conformal class $[g]$ we define the set
\begin{align} \label{Cp}
\Ca = \{ g_u = e^{-2u}g \in \Cm\ :\ \mbox{(A) holds for }g_u \}.
\end{align}
Although this set seems unrelated to any `cone of ellipticity' for the equation (\ref{vkY}), we will see that it carries all the information needed to verify the desired variational properties of $F$.  In addition, when the manifold is LCF we will show that (A) holds for any metric in the positive $n/2$-cone (see Proposition \ref{andrewstype} below).  This is further evidence that this condition is natural. 

To summarize, our first main result is the following:

\begin{thm} \label{functionalprops} Let $(M^{2m}, g)$, be a closed, even-dimensional Riemannian manifold.\\

\noindent $(i)$ There exists a
functional $F : [g] \to \mathbb R$ such that the critical points of $F$ satisfy
(\ref{vkY}). \\

\noindent $(ii)$ $F : \Ca \rightarrow \mathbb{R}$ is geodesically convex with respect to the Riemannian structure defined by
(\ref{vkmetric}). \end{thm}

As discussed above, in the LCF setting if the cone $\Gamma_{n/2}^{+}$ is non-empty then (up to a double cover) the manifold is the round sphere.  In this case we show that $\Gamma_{n/2}^{+}$  lies in $\CC^+_{(A)}$, yielding a new perspective on the uniqueness of solutions originally established by Viaclovsky \cite{JeffAMS}:

\begin{thm} \label{LCF}   Let $(S^{2m}, g_0)$ be the round sphere of dimension $2m$.  Then
\begin{align} \label{cones}
\Gamma_{m}^{+}([g_0]) = \{ g_u = e^{-2u}g_0\ :\ A_u \in \Gamma_{m}^{+} \} \subseteq \Ca.
\end{align}
In particular, (A) holds and $F : \Gamma_{m}^{+}([g_0]) \rightarrow \mathbb{R}$ is geodesically convex.  Also, critical points of $F$ are given by the round metric and its image under the conformal group \cite{JeffAMS}.
\end{thm}

The formal framework given by Theorem \ref{functionalprops} naturally
suggests that solutions to (\ref{vkY}) are unique up to scaling:

\begin{conj} Let $(M^{2m}, g)$, be a closed, even-dimensional Riemannian manifold such
that $\Ca \neq \emptyset$.  If $(M^{2m}, g)$ is not conformally equivalent to $(S^{2m},
g_0)$, then there exists a unique $u \in C^{\infty}(M)$ such that
$\int_M u dV_g = 0$ and $v_m(e^{-2u} g) = \bv$.
\end{conj}

\noindent To verify this along the lines of \cite{GS_Surfaces, GS2} would require existence results for the geodesic problem, which we do not address here.

%
%
%
%
As in our previous work, we also consider the gradient flow of $F$ with respect to the metric (\ref{vkmetric}).  Using properties of
$F$ which follow from its construction, one can show that the negative gradient flow (written as an evolution equation for the conformal
factor) is given by
\begin{align} \label{iskflow}
\dt u =&\ 1 - \frac{\bar{v}}{v_m(g_u)}.
\end{align}
For simplicity we will refer to this as the \emph{inverse $v_m$-flow}.  In
\S \ref{gradflowsec} we observe a number of formal
properties for solutions of this flow in the set $\Ca$.  These are in line with the properties established for
Calabi flow in relation to the Mabuchi metric in K\"ahler geometry
\cite{CalabiChen, Donaldson, MabuchiSymp, Semmes}:

\begin{thm} \label{iskflowprops} Let $(M^{2m}, g)$ be a closed, even-dimensional
Riemannian manifold, and suppose $u = u(t)$ is a solution to inverse $v_m$-flow with $g_u = e^{-2 u}g$ and $g_u \in \Ca$.
\begin{enumerate}

\item  Then $F$ is convex along the flow:
\begin{align*}
\frac{d^2}{dt^2} F[u] \geq 0.
\end{align*}

\item The following entropy estimate holds along the flow:
\begin{align*}
\frac{d}{dt} \int_M v_m \log v_m dV_u \leq&\ 0.
\end{align*}

\end{enumerate}
\end{thm}

Despite the excellent formal properties of the (negative) gradient flow, we lack a short-time existence result for solutions.  This is related to the
more general issues involved in the study of the $v_k$-Yamabe problem for $k \geq 3$, and are discussed in \cite{ChangFang} and \cite{Graham}.  A fundamental difficulty is identifying a useful notion of ellipticity.  For example,
there is no obvious way to conclude that the linearized operator $L$ is positive definite by imposing sign conditions on the $v_k$'s.  In the case where $k = m = n/2$, it would be
interesting to see whether the validity of the Andrews-type inequality (A) could be used in place of an algebraic condition.   Indeed one of the motivations for this paper is to attempt to put the $v_k$-Yamabe problem (at least for $k = n/2$) in the framework of a convex variational problem defined on a metric space, with the eventual goal of proving the existence of critical points without resorting to elliptic theory (i.e., `pointwise' methods).

\subsection*{Acknowledgements} The authors would like to thank Robin Graham for
several useful conversations, and Fedor Petrov for providing help with
Proposition \ref{crooshineq}.

%
%
%
%
%
%
\section{Metric and connection}  \label{MC}

We begin with the following variational formula for $v_m$ shown by
Graham
\cite{Graham}.  Note that we have changed convention from that paper and
parameterize conformal metrics via $g_u = e^{-2u} g$.

\begin{thm} \label{grahamthm} (\cite{Graham} Theorem 1.5) Given $(M^{2m}, g)$ be a
closed, even-dimensional Riemannian manifold, and let $u = u(t)$ be a
one-parameter family of conformal factors such that $u(0) = 0$, and
\begin{align*}
\left. \frac{d}{dt} u \right|_{t=0} = \dot{u}.
\end{align*}
There exists a natural tensor $L$ such that
\begin{align} \label{dsig}
\left. \frac{d}{dt} v_m(e^{-2u(t)}g) \right|_{t=0} = n \dot{u} v_m(g) + \N_i \left( L_{}^{ij} \N_j
\dot{u} \right).
\end{align}
\end{thm}

\begin{cor} \label{confinfcor} Given $(M^{2m},g)$ as above, the quantity
\begin{align*}
v := \int_M v_m(g) dV
\end{align*}
is a conformal invariant.
\end{cor}

The tensor $L$ above is derived in \cite{Graham} using the ``ambient
metric construction," and the reader should consult that work for full details.
The relevant point for us is that the linearization is a divergence-form
operator.  Moreover, in the case that the metric $g$ is locally conformally
flat, $L$ is the $k-1$ Newton transformation of the Schouten tensor.

 We can now record some basic formal properties of the inner product defined by (\ref{vkmetric}).

\begin{defn} \label{sigmakmetric}  Let $(M^{2m}, g)$ be
a closed Riemannian manifold of dimension $2m$.  The \emph{$v_m$-metric} is the formal
Riemannian metric defined for $u \in \Cm$, $\ga,\gb \in T_u \Cm \cong
C^{\infty}(M)$ via
\begin{align} \label{metricdef}
\IP{\ga,\gb}_{u} =  \int_M \ga \gb v_m(g_u) dV_u.
\end{align}
Moreover, given a path $u = u(t)$ in $\Cm$ and a one-parameter family of
tangent vectors $\alpha = \alpha(t)$ with $\ga(t) \in T_{u(t)} \Cm$, let
\begin{align} \label{connectiondef}
\frac{D}{\del t} \ga := \frac{\partial}{\partial t}{\ga} - v_m^{-1} \IP{L, \N \ga
\otimes \N \frac{\partial}{\partial t}}
\end{align}
denote the directional derivative along the path $u(t)$.
\end{defn}

\begin{lemma}  \label{metcomp} The connection
defined by (\ref{connectiondef}) is metric compatible and torsion free.
\begin{proof} First we check metric compatibility. Using Theorem \ref{grahamthm} we compute
\begin{align*}
\frac{d}{dt} \IP{\ga_t, \gb_t}_{u_t} =&\ \frac{d}{dt} \int_M \ga \gb v_m(g_u)
dV_u\\
=&\ \IP{\frac{\partial}{\partial t}\alpha , \gb} + \IP{\ga,\frac{\partial}{\partial t}\beta} + \int_M \ga \gb
\N_i \left( L^{ij} \N_j \dtu \right) dV_u\\
=&\ \IP{\frac{\partial}{\partial t}\alpha , \gb} + \IP{\ga,\frac{\partial}{\partial t}\beta} - \int_M \IP{L, \left(
\ga \N \gb + \gb \N \ga \right) \otimes \N \dtu} dV_u\\
=&\ \IP{ \frac{D}{\del t} \ga,\gb} + \IP{\ga,\frac{D}{\del t} \gb}.
\end{align*}
Next, to compute the torsion, let $u = u(s,t)$ be a two-parameter family of
conformal factors.
Then
\begin{align*}
\frac{D}{\del s} \frac{\del u}{\del t} - \frac{D}{\del t} \frac{\del u}{\del
s} =&\ \frac{\del^2 u}{\del s \del t} - v_m^{-1} \IP{L, \N
\frac{\del u}{\del s} \otimes \N \frac{\del u}{\del t}} - \frac{\del^2 u}{\del s
\del t} + v_m^{-1} \IP{L, \N \frac{\del u}{\del t} \otimes \N
\frac{\del u}{\del s} }\\
=&\ 0.
\end{align*}
The lemma follows.
\end{proof}
\end{lemma}

Next, we observe some properties of lengths of curves
and distances in the $v_m$-metric.

\begin{defn}  Given a path $u : [a,b] \to \Cm$, the
\emph{length of $u$} is
\begin{align} \label{lengthdef}
\mathcal \ell[u] := \int_a^b \IP{\ga,\gb}^{\frac{1}{2}} dt = \int_a^b \left[ \int_M
\left(
\frac{\del u}{\del t} \right)^2 v_m(g_u) dV_u \right]^{\frac{1}{2}}dt.
\end{align}
A curve is a \emph{geodesic} if it is a critical point for $\ell$.
\end{defn}

\begin{lemma} \label{geodL} A curve $u = u(t) \in \Cm$ is a geodesic if
and only if
\begin{align} \label{pathgeod}
\frac{\partial^2}{\partial t^2}u - v_m^{-1} \IP{L, \N \frac{\partial }{\partial t}u \otimes \N \frac{\partial }{\partial t}u} = 0.
\end{align}
\begin{proof} Formally, by Lemma \ref{metcomp} the connection is indeed the
Riemannian connection and so a curve is a
geodesic if and only if
\begin{align*}
0 =&\ \frac{D}{\del t} \frac{\del u}{\del t} = \frac{\partial^2 }{\partial t^2}u - v_m^{-1}
\IP{L, \N \frac{\partial }{\partial t}u \otimes \N \frac{\partial }{\partial t}u}.
\end{align*}
This can also be derived by directly taking the first variation of the length
functional.
\end{proof}
\end{lemma}

\begin{rmk}  There is a canonical isometric
splitting of the tangent space at each $g_u \in \Cm$ with respect to
the $v_m$-metric.  In particular, the real line $\mathbb R \subset T_u \Cm$
given by constant functions is orthogonal to
\begin{align*}
T^0_u \Cm := \left\{ \ga\ |\ \int_M \ga v_m dV_u = 0 \right\}.
\end{align*}
\end{rmk}

\noindent In the next lemma we show two basic properties of geodesics, namely
that they preserve this isometric splitting, and are automatically parameterized
with constant speed.

\begin{lemma}  \label{metricsplit} Let $u = u(t)$ be a
solution to (\ref{pathgeod}).  Then
\begin{align*}
\frac{d}{dt} \int_M  \frac{\partial u}{\partial t}  v_m dV_u =&\ 0,\\
\frac{d}{dt} \int_M (\frac{\partial u}{\partial t})^2 v_m dV_u =&\ 0.
\end{align*}
\begin{proof} First we differentiate
\begin{align*}
\frac{d}{dt} \int_M  \frac{\partial u}{\partial t}  v_m dV_u =&\ \int_M \left(  \frac{\partial^2 u }{\partial t^2}
v_m +
\frac{\partial u}{\partial t} \N_i (L^{ij} \N_j \frac{\partial}{\partial t}u) \right) dV_u\\
=&\ \int_M \left(  \frac{\partial^2 }{\partial t^2}u - v_m^{-1} \IP{L, \N \frac{\partial u}{\partial t}
\otimes \N
\frac{\partial}{\partial t}u} \right) v_m dV_u\\
=&\ 0.
\end{align*}
Next
\begin{align*}
\frac{d}{dt} \int_M ( \frac{\partial u}{\partial t})^2 v_m dV_u =&\ \int_M \left[ 2 v_m
\frac{\partial^2 u}{\partial t^2} \frac{\partial u}{\partial t} + (\frac{\partial u}{\partial t})^2 \N_i (L^{ij} \N_j \frac{\partial u}{\partial t}) \right] dV_u\\
=&\ 2 \int_M v_m \frac{\partial u}{\partial t} \left[ \frac{\partial^2 u }{\partial t^2} -
v_m^{-1} \IP{L, \N
\frac{\partial u}{\partial t} \otimes \N \frac{\partial u}{\partial t}} \right] dV_u\\
=&\ 0.
\end{align*}
\end{proof}
\end{lemma}

One expects other formal aspects of the metric space structure established in \cite{GS2} to extend to this setting as well.  For instance, it is natural to expect nonpositive curvature of this metric.  Moreover, formal arguments suggest that the distance function induced by this Riemannian structure should be nondegenerate.  Proofs of these statements should follow along similar lines to \cite{GS2} but as we have no concrete application we do not pursue this here.

\section{The functional $F$ and geodesic convexity}

In this subsection we generalize Brendle-Viaclovsky's derivation \cite{BV} of
the conformal
primitive for the equation
\begin{align*}
v_m(g_u) = const.
\end{align*}
To begin we define a one-form $\alpha_u : T_u [g] \rightarrow \mathbb{R}$ on our given
conformal class via
\begin{align*}
\ga_u(\phi) :=&\ \int_M \phi v_m(g_u) dV_u.
\end{align*}

\begin{lemma} \label{BVlemma} The $1$-form $\ga$
is exact.
\begin{proof} Suppose that $u = u(s,t)$ is a two-parameter family of conformal
factors, and compute
\begin{align*}
\frac{d}{ds} \ga \left( \frac{\del u}{\del t} \right) =&\ \int_M \frac{\del^2
u}{\del s \del t} v_m (g_u) dV_u + \int_M \frac{\del u}{\del t}
\frac{\del}{\del s} \left( v_m(g_u) dV_u \right)\\
=&\ \int_M \frac{\del^2 u}{\del s \del t} v_m(g_u) dV_u + \int_M \frac{\del
u}{\del t} \N_i \left( L^{ij} \N_j \frac{\del u}{\del s} \right) dV_u\\
=&\ \int_M \frac{\del^2 u}{\del s \del t} v_m(g_u) dV_u - \int_M
L^{ij} \N_i \frac{\del u}{\del s} \otimes \N_j \frac{\del u}{\del t}
dV_u.
\end{align*}
This expression is manifestly symmetric in $s$ and $t$, thus $\ga$ is closed.
Since the space of conformal factors is contractible, this implies that $\ga$ is
exact.
\end{proof}
\end{lemma}

\begin{prop}  \label{CYvar} Let $(M^{2m}, g)$ be a closed, even-dimensional Riemannian manifold.  Then there is a functional $F : [g] \rightarrow \mathbb{R}$ such that
if $u = u(t) : (-\epsilon , \epsilon) \rightarrow [g]$ is a path
with $u(0) = u$ and $\frac{d}{dt}u(t)|_{t=0} = \dot{u}$, then
\begin{align} \label{Lvar}
\frac{d}{dt} F[u(t)]\big|_{t=0} =&\ \int_M \dot{u} \left[ - v_m(g_u) + \bar{v}
\right] dV_u.
\end{align}
\begin{proof} Since the $1$-form $\ga$ is exact by Lemma \ref{BVlemma}, there
exists a function $E : [g] \to \mathbb R$ such that $d E = \ga$.  We thus set
$F[u] = E[u] - \frac{\bv}{n} \log \int_M dV_u$, and the result follows.
\end{proof}
\end{prop}

\begin{prop} \label{CYgeodconv} Let $(M^{2m}, g)$ be a closed, even-dimensional Riemannian manifold.  Then $F : \Ca \rightarrow \mathbb{R}$ is geodesically convex.

\begin{proof} Let $u = u(t)$ be a geodesic.  Using Lemma \ref{metricsplit} and the inequality (A) we have
\begin{align*}
\frac{d^2}{dt^2} F[u(t)] =&\ \frac{d}{dt} \int_M \frac{\partial u}{\partial t} \left[ - v_m +
\bar{v} \right] dV_u\\
=&\  v \frac{d}{dt} \int_M  ( \frac{\partial}{\partial t}) V_u^{-1} dV_u\\
=&\ v \int_M \left[\frac{\partial^2 u}{\partial t^2} V_u^{-1}  + V_u^{-2} \frac{\partial u}{\partial t} \left(
\int_M n ( \frac{\partial u}{\partial t}) dV_u \right) - n V_u^{-1} ( \frac{\partial u}{\partial t})^2 \right] dV_u\\
=&\ v V_u^{-1} \left[ \int_M v_m^{-1} \IP{L,
\N \frac{\partial u}{\partial t} \otimes \N \frac{\partial u}{\partial t}} dV_u - n \left( \int_M ( \frac{\partial u}{\partial t})^2 dV_u - V_u^{-1}
\left( \int_M
\frac{\partial u}{\partial t} dV_u \right)^2 \right) \right]\\
\geq&\ 0.
\end{align*}
\end{proof}
\end{prop}

\begin{proof}[Proof of Theorem \ref{functionalprops}]
Part $(i)$ follows from Proposition \ref{CYvar}, while part $(ii)$ follows from Proposition \ref{CYgeodconv}.
\end{proof}

\section{The locally conformally flat case} \label{poincare}

As we pointed out in the Introduction, it follows from Kuiper's Theorem and the work of Guan-Viaclovky \cite{GuanVia} that if $(M^{2m}, g)$ is a closed, even-dimensional LCF manifold with
\begin{align*}
\Gamma_m^{+}([g]) = \{ g_u = e^{-2u}g\ :\ A_u \in \Gamma_m^{+} \} \neq \emptyset,
\end{align*}
then $(M^{2m}, g)$ must be conformally equivalent to the round sphere or real projective space.  In this section we show that for the round sphere $(S^{2m},g_0)$ the cone $\Gamma_m^{+}([g])$ is contained in $\Ca$; i.e., the Andrews-type inequality holds.  As an immediate consequence, $F : \Gamma_m^{+}([g_0]) \rightarrow \mathbb{R}$ is geodesically
convex.  It was shown by Viaclovsky \cite{JeffAMS} that all solutions of $\sigma_m(A_u) = const.$ are given by conformal metrics with $g_u = \varphi^{\ast}g_0$ for some conformal transformation $\varphi : S^{2m} \rightarrow S^{2m}$, therefore giving us a complete variational description in this case.

The proof is an application of a closely related inequality of Andrews \cite{Andrews},
exploiting certain inequalities relating elementary symmetric polynomials.  We
begin with the inequality of Andrews:

\begin{prop} \label{andrewsineq} (Andrews \cite{Andrews}, cf. \cite{CLN} pg.
517) Let $(M^n, g)$ be a closed
Riemannian manifold with
positive Ricci curvature.  Given $\phi \in C^{\infty}(M)$ such that $\int_M \phi
dV = 0$, then
\begin{align*}
\frac{n}{n-1} \int_M \phi^2 dV \leq&\ \int_M \left( \Rc^{-1} \right)^{ij} \N_i
\phi
\N_j \phi dV,
\end{align*}
with equality if and only if $\phi \equiv 0$ or $(M^n, g)$ is isometric to the
round sphere.
\end{prop}

To show how this inequality implies (A) for metrics in $\Gamma_m^{+}([g])$, we use an argument that was shown to us by Petrov
\cite{Petrov}.  Given $\gl_1,\dots,\gl_n \in \mathbb R$, let
\begin{align*}
\gs_k (\gl) =&\ \sum_{1 \leq i_1 < i_2 < \dots < i_k \leq n} \gl_{i_1}\dots
\gl_{i_k},\\
\gs_{k;i}(\gl) =&\ \gs_k(\gl)|_{\gl_{i} = 0},\\
\til{\gs}_k(\gl) =&\ {n \choose k}^{-1} \gs_k(\gl_1,\dots,\gl_n).
\end{align*}

\begin{lemma} (cf. \cite{LinT}) With the notation above one has
\begin{align} \label{elemineq}
\gs_k(\gl) =&\ \gs_{k;i}(\gl) + \gl_i \gs_{k-1;i}(\gl)\\
\left[ {n \choose k}^{-1} \gs_k \right]^{\frac{1}{k}} \leq&\ \left[ {n \choose
l}^{-1} \gs_l \right]^{\frac{1}{l}}, \qquad k \geq l \geq 1 \label{maclaurin}
\end{align}
\end{lemma}

\begin{lemma} \label{vieta} Given $\gl_1, \dots, \gl_n \in \mathbb R$, let $F(x)
= \prod_{i=1}^n (x - \gl_i)$, and suppose $F'(x) = n \prod_{i=1}^{n-1} (x -
\mu_i)$.  Then
\begin{align*}
\til{\gs}_k(\gl_1,\dots,\gl_n) = \til{\gs}_k(\mu_1,\dots,\mu_{n-1})
\end{align*}
\begin{proof} This follows directly from the classical Vieta formulas.
\end{proof}
\end{lemma}

\begin{prop} \label{crooshineq} (\cite{Petrov}) Fix $n = 2m$.  Given $\lambda =
(\lambda_1, \dots, \lambda_n)$, let
\begin{align*}
A_{\lambda} =&\ \frac{1}{n-2} \left[ \lambda - \frac{\sigma_1(\lambda)}{2(n-1)}
(1,\dots,1) \right]
\end{align*}
Given $\lambda$ such that $A_{\lambda} \in \Gamma_m^{+}$, for all $i$ one has that
\begin{align} \label{f:crooshineq10}
(n-1) \sigma_m(A_{\lambda}) \leq \lambda_i \sigma_{m-1;i}(A_{\lambda}).
\end{align}
\begin{proof} Let $A_{\gl} = (a_1,\dots,a_n)$.  We fix $i=n$ for convenience, no
ordering on the $\gl_i$ is assumed.  Note that
\begin{align*}
\gl_n = (n-2) a_n + (a_1 + \dots + a_n) = (n-1) a_n + (a_1 + \dots + a_{n-1}).
\end{align*}
Applying this and (\ref{elemineq}) we see that the required inequality is
equivalent to
\begin{align*}
(n-1) \gs_m(A_{\gl}) \leq&\ \gl_n \gs_{m-1;n}(A_{\gl})\\
=&\ \left[ (n-1) a_n + \gs_1(a_1,\dots,a_{n-1}) \right]
\gs_{m-1}(a_1,\dots,a_{n-1})\\
=&\ (n-1) \left[ \gs_m(a_1,\dots,a_n) - \gs_m(a_1,\dots,a_{n-1}) \right] +
\gs_1(a_1,\dots,a_{n-1}) \gs_{m-1}(a_1,\dots,a_{n-1}),
\end{align*}
hence we see that it suffices to show that
\begin{align*}
(n-1) \gs_m(a_1,\dots,a_{n-1}) \leq \gs_1(a_1,\dots,a_{n-1})
\gs_{m-1}(a_1,\dots,a_{n-1}).
\end{align*}
Written in terms of normalized functions, since $m=\frac{n}{2}$ this is
equivalent to
\begin{align} \label{ci10}
\til{\gs}_m(a_1,\dots,a_{n-1}) \leq \til{\gs}_1(a_1,\dots,a_{n-1})
\til{\gs}_{m-1}(a_1,\dots,a_{n-1}).
\end{align}
Now let $f(x) = \prod_{i=1}^{n-1} (x - a_i)$, and let $g(x) = f^{(m-1)}(x)$.
Note that $g$ is a polynomial of degree $m$ with real roots $b_1\dots,b_m$.  By
Lemma \ref{vieta} we see that (\ref{ci10}) is equivalent to
\begin{align*}
\gs_m(b_1,\dots,b_m) =&\ \til{\gs}_m(b_1,\dots,b_m)\\
\leq&\ \til{\gs}_1(b_1,\dots,b_m) \til{\gs}_{m-1}(b_1,\dots,b_m)\\
=&\ m^{-2} \gs_1(b_1,\dots,b_m) \gs_{m-1}(b_1,\dots,b_m).
\end{align*}
If all $b_i > 0$, this follows directly from Maclaurin's inequality,
(\ref{maclaurin}).

Now suppose that $b_n \leq 0$.  Observe that $f(x)(x-a_n)$ has positive
coefficients, and $f(x)$ has only real roots. It follows from Rolle's theorem
that $(f(x)(x-a_n))^{(m)}$ has only real roots and they must be positive. Hence
$h(x) := (x-a_n)g'(x)+mg(x)$ has $m$ positive roots. Thus $(g(x) (x-a_n)^m)'$
has a root $a_n$ of multiplicity $m-1$ and $m$ positive roots.   If $b_n\leq 0$,
then by Rolle's theorem $(g(x)\cdot (x-a_n)^m)'$ has a root between $a_n$ and
$b_m$ (or has a root $a_n$ of multiplicity at least $m$ if $a_n=b_m$), which
contradicts the above discussion.  Analogously, if $a_n\geq 0$ but $b_i\leq 0$
for some $i\neq m$, we get a negative root of $(g(x)\cdot (x-a_n)^m)'$, which is
again a contradiction.  So, $a_n>0$ and $b_1,\dots,b_{m-1}>0$. Since $h(x)$ has
$m$ positive roots and positive leading coefficient, we have $h(0)(-1)^{m}>0$.
On the other hand, $g(0)(-1)^m=b_1\dots b_m<0$. Thus $(-1)^{m-1}a_n g'(0)>0$,
i.e. $\sigma_{m-1}(b_1,\dots,b_m)>0$.  It follows that $b_1+\dots+b_m>0$, thus
$(b_
1+\dots+b_m)\sigma_{m-1}(b_1,\dots,b_m)>0>m^2 b_1\dots b_m$ as desired.

\end{proof}
\end{prop}

\begin{prop} \label{andrewstype} Let $(M^{2m}, g)$ be a closed, even-dimensional Riemannian manifold
such that $A_g \in \gG_m^+$.  Given $\phi \in C^{\infty}(M)$ one has
\begin{align} \label{weightedpoincare}
n \left[
\int_M \phi^2 dV_g - V_g^{-1} \left( \int_M \phi
dV_g \right)^2 \right]\leq \int_M \frac{1}{\gs_{m} (A_g)} T_{m-1} (A_g)^{ij}
\N_i \phi \N_j \phi dV_g.
\end{align}
\begin{proof} Choosing a point $p \in M$ choose a basis for $T_p M$ such that the Ricci tensor is diagonalized, with eigenvalues $\gl_i$.  Since, in the notation of Proposition \ref{crooshineq}, the matrix $A_{\gl}$ corresponds to the Schouten tensor, which is also diagonalized, the inequality (\ref{f:crooshineq10}) can be rearranged to imply the matrix inequality
\begin{align*}
(n-1) \Rc^{-1} \leq \frac{1}{\gs_m(A_g)} T_{m-1}(A_g).
\end{align*}
The proposition thus follows from Proposition \ref{andrewsineq}.
\end{proof}
\end{prop}

\begin{proof}[Proof of Theorem \ref{LCF}]  In the LCF setting it was shown by Graham-Juhl \cite{GrahamJuhl} that $\sigma_m(A_g) = v_m(g)$ and consequently $T_{m-1}(A) = L$.  Therefore, Theorem \ref{LCF} follows from Proposition \ref{andrewstype} and Theorem \ref{functionalprops}.
\end{proof}

%
%
%
%
\section{The  gradient flow of the functional $F$}
\label{gradflowsec}

In this section we derive the $v_m$-metric gradient flow of the $F$ functional
and many formal properties of it related to the $v_m$-metric, culminating
in
the proof of Theorem \ref{iskflowprops}.

\begin{lemma}  With respect to the
$v_m$-metric, one has
\begin{align*}
\left[\grad F \right]_u  =&\  -1 + \frac{ \bar{v}}{v_m}.
\end{align*}
\begin{proof} Using Proposition \ref{CYvar} and the definition of the $\sigma_k$
metric we have
\begin{align*}
\frac{d}{dt} F[u] =&\ \int_M \frac{\partial}{\partial t}u \left[ - v_m(g_u) + \bar{v}
\right] dV_u\\
=&\ \int_M \frac{\partial}{\partial t}u \left[ - 1 + \frac{ \bar{v}}{v_m (A_u)} \right]
v_m(g_u) dV_u\\
=&\ \IP{ \frac{\partial}{\partial t}u, -1 + \frac{ \bar{v}}{v_m(g_u)}}_u.
\end{align*}
\end{proof}
\end{lemma}

\begin{defn}  We say that the path
$u = u(t)$ of conformal factors is a
solution to \emph{inverse $v_m$-flow} if $g_u = e^{-2u}g \in \Ca$
\begin{align*}
\dt u =&\  - \left[\grad F \right]_u = 1 - \frac{\bar{v}}{v_m(g_u)}.
\end{align*}
\end{defn}

We now establish a number of monotonicity properties for the
inverse $v_m$-flow.  In Proposition \ref{Fflowconvexity} we show that
the $F$ functional is convex along a flow line. We then prove the monotonicity of entropy (Proposition \ref{entropymonotonicity}).
Finally, in Proposition \ref{lengthmonotonicity} we establish that the
length of curves is monotone nonincreasing when flowed along inverse $v_m$
flow.

\begin{prop}  \label{Fflowconvexity} Given $u = u(t)$
a solution to the inverse
$v_m$ flow, one has
\begin{align*}
\frac{d^2}{dt^2} F[u_t] \geq&\ 0.
\end{align*}
\begin{proof} Using (\ref{Lvar}) we have for a solution to inverse $v_m$ flow
\begin{align*}
\frac{d}{dt} F =&\ - \int_M \left( 1 - \frac{\bar{v}}{v_m} \right)^2
v_m dV_u\\
=&\ - \int_M \left[1 - \frac{2 \bar{v}}{v_m} + \frac{\bar{v}^2}{v_m^2}
\right] v_m dV_u\\
=&\ - v + 2 v - \bar{v}^2 \int_M \frac{1}{v_m} dV_u.
\end{align*}
Hence using the variational formula (\ref{dsig}) we have
\begin{align*}
\frac{d^2}{dt^2} F =&\ \frac{d}{dt} \left[ - \frac{v^2}{V_u^2} \int_M
\frac{1}{v_m} dV_u \right]\\
=&\ \frac{v^2}{V_u^3} \left[ 2 \frac{d}{dt} V_u \int_M \frac{1}{v_m}
dV_u - V_u \frac{d}{dt} \int_M \frac{1}{v_m} dV_u \right]\\
=&\ \frac{v^2}{V_u^3} \left[ 2 \left( - n V_u + n \bar{v}
\int_M \frac{1}{v_m} dV_u \right) \int_M \frac{1}{v_m} dV_u \right.\\
&\ \qquad - V_u \int_M n \left(- 1 + \frac{\bar{v}}{v_m} \right)
\frac{1}{v_m} dV_u\\
&\ \left. \qquad - V_u \int_M \left( - \frac{n}{v_m} + \frac{n
\bar{v}}{v_m^2} + \frac{\bar{v}}{v_m^2} \N_i (L^{ij} \N_j \frac{1}{v_m}) \right)
dV_u \right]\\
=&\ \frac{v^2}{V_u^3} \left[ 2n \bar{v} \left( \int_M
\frac{1}{v_m} dV_u \right)^2 - 2n v \int_M \frac{1}{v_m^2} dV_u -
v \int_M \frac{1}{v_m^2} \N_i (L^{ij} \N_j \frac{1}{v_m})  dV_u
\right]\\
=&\ \frac{v^2}{V_u^3} \left[ 2n \bar{v} \left( \int_M
\frac{1}{v_m} dV_u \right)^2 - 2n v \int_M \frac{1}{v_m^2} dV_u + 2
v \int_M \frac{1}{v_m} \IP{L, \N \frac{1}{v_m} \otimes \N
\frac{1}{v_m}} dV_u \right]\\
=&\ \frac{2 v^3}{V_u^3} \left[ - n \int_M \left( \frac{1}{v_m} -
V_u^{-1} \int_M \frac{1}{v_m} \right)^2 + \int_M \frac{1}{v_m}
\IP{L, \N \frac{1}{v_m} \otimes \N \frac{1}{v_m}} dV_u \right]\\
\geq&\ 0,
\end{align*}
where the last line follows from applying (A).
\end{proof}
\end{prop}

The next result is a monotonicity result for the entropy along the flow:

\begin{prop} \label{entropymonotonicity} Given $u = u(t)$
a solution to
inverse
$v_m$-flow, one has
\begin{align*}
\frac{d}{dt} \int_M v_m \log v_m dV_u =&\ - \bar{v} \int_M
 \IP{L, \N \frac{1}{v_m} \otimes \N \frac{1}{v_m}} \leq 0.
\end{align*}
\begin{proof} Again using the variational formula for $v_m$,
\begin{align*}
\frac{d}{dt} & \int_M v_m \log v_m dV_u\\
=& \int_M \frac{\partial}{\partial t}( v_m \log v_m) dV_u + \int_M v_m \log v_m \frac{\partial}{\partial t}dV_u \\
=&\ \int_M \left[1 + \log v_m \right]
\frac{\partial}{\partial t}v_m dV_u - \int_M v_m \log v_m (n \frac{\partial}{\partial t}u) dV_u\\
=&\ \int_M \left[1 + \log v_m \right] \left[ - \N_i (L^{ij} \N_j
\frac{\bar{v}}{v_m}) - n \left[  v_m - \bar{v}_k \right] \right] + n
\int_M
v_m \log v_m \left[ 1 - \frac{\bar{v}_k}{v_m} \right]\\
=&\ - \bar{v} \int_M  \IP{L, \N \frac{1}{v_m} \otimes \N
\frac{1}{v_m}}.
\end{align*}
\end{proof}
\end{prop}

\begin{proof}[Proof of Theorem \ref{iskflowprops}]  This follows immediately from the preceding two propositions.
\end{proof}

We include a final proposition since it provides an interesting parallel with a result of Calabi-Chen \cite{CalabiChen} on the length of paths under the Calabi flow in K\"ahler geometry:

\begin{prop}  \label{lengthmonotonicity} Let
$u = u(s,t)$ be a two-parameter family of conformal factors such that $g_u = e^{-2u}g \in \Ca$ for $s \in [0,1], t \in [0,T)$.  We also assume
that for all $s \in [0,1]$, the path $t \mapsto u(\cdot,t)$ is a solution to inverse $v_m$-flow.
Then
\begin{align*}
\frac{d}{dt} \ell( u(\cdot, t)) \leq 0.
\end{align*}
\begin{proof} We directly compute
\begin{align*}
\frac{d}{dt} \ell( u(\cdot, t)) =&\ \frac{d}{dt} \int_0^1 \left[ \int_M \left(
\frac{\del u}{\del s} \right)^2 v_m(g_u) dV_u \right]^{\frac{1}{2}}ds \\
=&\ \frac{1}{2} \int_0^1 \brs{\frac{\partial u}{\partial s} }_u^{-1} \int_M \left[ 2 \frac{\partial^2 u}{\partial s \partial t} \frac{\partial u}{\partial s} v_m +
( \frac{\partial u}{\partial s})^2 \N_i (L^{ij} \N_j \frac{\partial u}{\partial t}) \right] dV_u ds.
\end{align*}
Now we compute
\begin{align*}
\frac{\partial^2 u}{\partial s \partial t} =&\ \frac{\del}{\del s} \left[ 1 - \frac{\bar{v}}{v_m} \right]\\
=&\ \bar{v} \left[ \frac{1}{V_u v_m} \frac{\del}{\del s} V_u +
\frac{1}{v_m^2} \N_i (L^{ij} \N_j \frac{\del u}{\del s}) + n \frac{\partial u}{\partial s}
v_m^{-1} \right]\\
=&\ \bar{v} \left[ - \frac{1}{V_u v_m} \int_M n \frac{\partial u}{\partial s} dV_u + \frac{1}{v_m^2} \N_i
(L^{ij} \N_j \frac{\del u}{\del s}) + n \frac{\partial u}{\partial s} v_m^{-1}
\right].
\end{align*}
Hence
\begin{align*}
\int_M & 2 \frac{\partial^2 u}{\partial s \partial t} \frac{\partial u}{\partial s}  v_m dV_u\\
=&\ 2 \bar{v} \int_M \left[ - n V_u^{-1}
v_m^{-1} \int_M \frac{\partial u}{\partial s}  dV_u +  v_m^{-2} \N_i (L^{ij} \N_j \frac{\partial u}{\partial s} ) + n v_m^{-1}
\frac{\partial u}{\partial s}  \right] \frac{\partial u}{\partial s}  v_m dV_u\\
=&\ 2n \bar{v} V_u^{-1} \left\{ V_u \int_M
(\frac{\partial u}{\partial s})^2 dV_u - \left[ \int_M \frac{\partial u}{\partial s}  dV_u \right]^2 \right\} + 2 \bar{v} \int_M
v_m^{-1} \frac{\partial u}{\partial s}  \N_i (L^{ij} \N_j \frac{\partial u}{\partial s} ) dV_u\\
=&\ 2n \bar{v} V_u^{-1} \left\{ V_u \int_M
(\frac{\partial u}{\partial s})^2 dV_u - \left[ \int_M \frac{\partial u}{\partial s}  dV_u \right]^2 \right\}\\
&\ - 2 \bar{v} \int_M \left[ \IP{L, v_m^{-1} \N \frac{\partial u}{\partial s}  \otimes
\N \frac{\partial u}{\partial s}  + \frac{\partial u}{\partial s}  \N \frac{\partial u}{\partial s}  \otimes \N v_m^{-1} } \right]dV_u.
\end{align*}
Also
\begin{align*}
\int_M (\frac{\partial u}{\partial s})^2 \N_i (L^{ij} \N_j \frac{\partial u}{\partial t} ) dV_u =&\ - 2 \int_M \frac{\partial u}{\partial s}  \IP{L, \N u_t
\otimes \N \frac{\partial u}{\partial s} } dV_u\\
=&\int_M (\frac{\partial u}{\partial s})^2 \N_i (L^{ij} \N_j ( \frac{\bv}{v_m} ) ) dV_u =&\ - 2 \int_M \frac{\partial u}{\partial s}  \IP{L, \N u_t
\otimes \N \frac{\partial u}{\partial s} } dV_u\\
=&\ 2 \bar{v} \int_M \IP{L, \frac{\partial u}{\partial s}  \N \frac{\partial u}{\partial s}  \otimes \N v_m^{-1}} dV_u.
\end{align*}
Collecting these calculations yields
\begin{align*}
\frac{d}{dt} \ell( u(\cdot, t))  =&\ - \int_0^1 \bar{v} \brs{\frac{\partial u}{\partial s} }_u^{-1} \left[
\int_M v_m^{-1} \IP{L, \N \frac{\partial u}{\partial s}  \otimes
\N \frac{\partial u}{\partial s} } dV_u - n \int_M (\frac{\partial u}{\partial s})^2 + n V_u^{-1} \left[ \int_M \frac{\partial u}{\partial s}  dV_u \right]^2
\right] ds\\
\leq&\ 0,
\end{align*}
where the last line follows from (A) (since $g_u \in \Ca$).
\end{proof}
\end{prop}


\begin{thebibliography}{s}
%
%
%
%

\bibitem{LinT} M. Lin, N.S. Trudinger, \emph{On some inequalities for elementary
symmetric functions}, Bull. Austral. Math. Soc. 50 (1994), no. 2, 317–326.

\bibitem{Andrews} B. Andrews, unpublished.


\bibitem{BG} T. Branson, P. Gilkey, J. Pohjanpelto, \emph{Invariants of locally conformally flat manifolds}, Trans. Amer. Math. Soc. 347 (1995), no. 3, 939-–953.





\bibitem{BV} S. Brendle, J. Viaclovsky, \emph{A variational characterization for
$\gs_{\frac{n}{2}}$}, Calc. Var. 20, 399-402 (2004).

\bibitem{CalabiChen} E. Calabi, X.X. Chen, \emph{The space of K\"ahler metrics
II}, J. Diff. Geom. 61 (2002), 173-193.

\bibitem{ChangFang} S.Y.A. Chang, H. Fang, \emph{A class of variational
functionals in conformal geometry}, Int. Math. Res. Not. (2008).


\bibitem{ChangYangMoserVol} S.Y.A. Chang, P. Yang, \emph{The inequality of Moser and Trudinger and applications to conformal geometry,
Dedicated to the memory of Jürgen K. Moser} Comm. Pure Appl. Math. 56 (2003), no. 8, 1135-–1150.

\bibitem{CLN} B. Chow, P. Lu, L. Ni, \emph{Hamilton's Ricci flow}, Lectures in Contemporary Mathematics, Science Press, Beijing.

\bibitem{Donaldson} S. K. Donaldson, \emph{Symmetric spaces, K\"ahler geometry
and Hamiltonion dynamics}, in Northern California Symplectic Geometry Seminar,
Amer. Math. Soc. Transl. Ser. 2, 196, Amer. Math. Soc., Providence, 1999, 13-33.

\bibitem{GrahamRNV}  C. R. Graham, \emph{Volume and area renormalizations for conformally compact Einstein metrics}, The Proceedings of the 19th Winter School "Geometry and Physics'' (Srní, 1999). Rend. Circ. Mat. Palermo (2) Suppl. No. 63 (2000), 31–42.

\bibitem{Graham} C. R. Graham, \emph{Extended obstruction tensors and
renormalized volume coefficients}, Adv. Math. 220 (2009), 1956-1985.

\bibitem{GrahamJuhl} C.R. Graham, A. Juhl, \emph{Holographic formula for
$Q$-curvature}, Adv. Math. 216 (2007), 841-853.


\bibitem{GuanVia} P. Guan, J. A. Viaclovsky, G. Wang, \emph{Some properties of the Schouten tensor and applications to conformal geometry},  Trans. Amer. Math. Soc. 355 (2003), no. 3, 925–933.

\bibitem{GS_Surfaces} M. J. Gursky, J. Streets, \emph{A formal Riemannian
structure on conformal classes and the inverse Gauss curvature flow}, preprint.

\bibitem{GS2} M. J. Gursky, J. Streets, \emph{A formal Riemannian structure on
conformal classes and uniqueness for the $\sigma_k$-Yamabe problem}, preprint.

\bibitem{MabuchiSymp} T. Mabuchi, \emph{Some symplectic geometry on compact
K\"ahler manifolds (1)}, Osaka J. Math 24 (1987), 227-252.

\bibitem{Petrov} F. Petrov,
\href{
http://mathoverflow.net/questions/233283/unusual-inequality-concerning-elementar
y-symmetric-functions}{\textsf{
http://mathoverflow.net/questions/233283}}

\bibitem{SchoenNU} R. Schoen, \emph{Variational theory for the total scalar
curvature functional for Riemannian metrics and related topics}. Topics in
calculus of variations (Montecatini Terme, 1987), 120–154,
Lecture Notes in Math., 1365, Springer, Berlin, 1989.

\bibitem{Semmes} S. Semmes, \emph{Complex Monge-Ampere equations and symplectic
manifolds}, Amer. J. Math. 114 (1992), 495-550.


\bibitem{JeffAMS} J. A. Viaclovsky, \emph{Conformally invariant Monge-Ampère equations: global solutions}, Trans. Amer. Math. Soc. 352 (2000), no. 9, 4371–4379.


\bibitem{JeffThesis} J.A. Viaclovsky, \emph{Conformal geometry, contact geometry, and the calculus of variations}, Duke. Math. J., Volume 101, No. 2 (2000), 283-316.


\end{thebibliography}
\end{document}